\documentclass[a4paper]{amsart}
\usepackage{hyperref}
\usepackage{pgffor}
\newcounter{SortListTotal}
\usepackage{amsmath,amsthm,wasysym,amssymb,amscd,mathrsfs}
\usepackage[all]{xy}
\usepackage{graphicx}
\usepackage{color}
\usepackage{combelow}
\usepackage{enumerate}

\newcommand{\sortitem}[2]{\expandafter\def\csname
SortListItem#1\endcsname{#2}\stepcounter{SortListTotal}}

\newcommand{\printsortlist}{\foreach\currentlistitem in{1,2,...,\value{SortListTotal}}{\item[\currentlistitem]\csname SortListItem\currentlistitem\endcsname}\setcounter{SortListTotal}{0}}

\newtheorem{theorem}{Theorem}
\newtheorem{thm}{Dirac Version of Theorem}
\newtheorem{prop}{Dirac Version of Proposition}
\newtheorem{cor}{Dirac Version of Corollary}
\newtheorem*{cor3}{Dirac Version of Corollary 3}
\newtheorem*{claim*}{Claim}
\newtheorem*{definition*}{Definition}

\newtheorem*{theorem*}{Theorem}

\newtheorem{corollary}{Corollary}
\newtheorem{definition}{Definition}
\newtheorem{example}{Example}
\newtheorem{lemma}{Lemma}
\newtheorem{proposition}{Proposition}
\newtheorem{remark}{Remark}

\newcommand{\diffto}{\xrightarrow{\raisebox{-0.2 em}[0pt][0pt]{\smash{\ensuremath{\sim}}}}}

\newcommand{\rmap}{\longrightarrow}

\newcommand{\ro}{\mathfrak{o}}

\newcommand{\w}{\omega}

\newcommand{\dd}{\mathrm{d}}

\newcommand{\R}{\mathbb{R}}
\newcommand{\PD}{\mathrm{PD}}

\begin{document}
\title{The homology class of a Poisson transversal}
\author{Pedro Frejlich}
\address{Universidade Federal do Rio Grande do Sul, Campus Litoral Norte, Rodovia RS 030, 11.700 -- Km 92, Emboaba -- Tramandaí, RS, CEP 95590-000}
\email{frejlich.math@gmail.com}
\author{Ioan M\u{a}rcu\cb{t}}
\address{Radboud University Nijmegen, IMAPP, 6500 GL, Nijmegen, The Netherlands}
\email{i.marcut@math.ru.nl}
\begin{abstract}
This note is devoted to the study of the homology class of a compact Poisson transversal in a Poisson manifold. For specific classes of Poisson structures, such as unimodular Poisson structures and Poisson manifolds with closed leaves, we prove that all their compact Poisson transversals represent non-trivial homology classes, generalizing the symplectic case. We discuss several examples in which this property does not hold, as well as a weaker version of this property, which holds for log-symplectic structures. Finally, we extend our results to Dirac geometry.
\end{abstract}
\maketitle
\tableofcontents
\setcounter{tocdepth}{1}

\subsection{Introduction}

A Poisson transversal of a Poisson manifold is an embedded submanifold which meets symplectic leaves transversally and symplectically. They go back to as far as P.~Dirac \cite{Dirac}, and play a key role in many constructions since the beginnings of Poisson Geometry \cite{Wein}, where small complementary transversals to symplectic leaves are used to describe the local structure. In many respects, they are the natural generalization of symplectic submanifolds: a version of Weinstein's symplectic tubular neighborhood theorem holds around Poisson transversals \cite{PT1}; in integrable Poisson manifolds, they give rise to symplectic subgroupoids \cite{CrFe2,PT1.5}.

In a symplectic manifold, any nonempty, compact symplectic submanifold represents a non-trivial homology class. This paper examines the analogous property in Poisson geometry. We say a Poisson manifold has the \textbf{HNPT property} (homologically non-trivial Poisson transversals) if its nonempty compact Poisson transversals represent non-trivial homology classes, and that it has the \textbf{weak HNPT property} if they represent non-trivial homology classes in their saturation. These properties are studied for specific classes of Poisson manifolds. Our main results are the following. Unimodular Poisson manifolds have the HNPT property (Theorem \ref{thm : main}); Poisson manifolds with closed leaves have the HNPT property (Theorem \ref{thm : all leaves closed implies HNPT}); log-symplectic Poisson manifolds have the weak HNPT property (Theorem \ref{thm : log weak HNPT}); in Theorem \ref{thm : closed and compact PTs and proper Poisson maps} we discuss functoriality of the HNPT property under proper Poisson maps, and, for completeness, we prove also that unimodularity is preserved under proper Poisson maps that are surjective submersions (Proposition \ref{pro : unimodular poisson}). In section \ref{sec : examples} we discuss examples of Poisson manifolds which fail to have the (weak) HNPT property.

Most of our results generalize to Dirac geometry, and the necessary adaptations are presented in section \ref{sec : HNPT Dirac}.

The natural homology class of a Poisson transversal lives in the homology twisted by the orientation bundle (even in the orientable case), and in Appendix A we give a brief outline of the needed theory with references.

In Appendix B we prove a result which is used to deduce Theorem \ref{thm : closed and compact PTs and proper Poisson maps}, namely that closed leaves of Poisson manifolds are embedded submanifolds (and generalizations of it). Although this fact seems well-known, a complete proof is hard to find in the literature.

\subsection*{Acknowledgements}
We thank Rui Loja Fernandes for suggesting the alternative proof of Corollary \ref{cor : no H 1}.

P.F.\ was supported by the Nederlandse Organisatie voor Wetenschappelijk Onderzoek (Vrije Competitie grant ``Flexibility and Rigidity of Geometric Structures'' 612.001.101) and by IMPA (CAPES-FORTAL project). I.M.\ was supported by the Nederlandse Organisatie voor Wetenschappelijk Onderzoek (Veni grant 613.009.031) and the National Science Foundation (grant DMS 14-05671).

\subsection{The homology class of a compact Poisson transversal}\label{sec : The homology class}

It will be convenient to use (co-)homology of a manifold $M$ twisted by the orientation bundle $\ro_M$ (see Appendix A for an overview of the used material and for notation); the resulting cohomology and homology groups will be denoted by:
\[H^{\bullet}(M,\ro_M),\ \ \ H_{\bullet}(M,\ro_M).\]

Given a Poisson manifold $(M,\pi)$ of dimension $m$, and an embedded submanifold $X\subset M$ of codimension $2q$ (Poisson transversals always have even codimension), the Poisson tensor induces a canonical section of the top exterior power of the normal bundle of $X$:
\[\mathrm{pr}(\pi^q|_X) \in \Gamma(\wedge^{2q} NX), \ \ \  \ NX:=TM|_X/TX,\]
where $\mathrm{pr}:\wedge TM|_X\to \wedge NX$ is the natural projection. The condition that $X$ be a Poisson transversal is equivalent to the section $\mathrm{pr}(\pi^q|_X)$ being nowhere-vanishing. Therefore, a Poisson transversal has a \textbf{canonical coorientation} induced by the section $\mathrm{pr}(\pi^q|_X)$. Thus, a compact Poisson transversal $X\subset (M,\pi)$ has a canonical \textbf{homology class} in homology with local coefficients in $\ro_M$ (see Appendix A):
\[[X]\in H_{m-2q}(M,\ro_M).\]

For instance, a symplectic manifold $(M,\omega)$ has a canonical orientation given by $\omega^{\mathrm{top}}$, which induces an isomorphism of flat bundles $\ro_M\simeq M\times \R$. Any compact Poisson transversal $X$ in $(M,\omega)$ (i.e.\ symplectic submanifold) has a non-trivial class in $H_{2p}(M,\R)\simeq H_{2p}(M,\ro_M)$, where $2p=\mathrm{dim}(X)$; this is simply because the closed form $\omega^{p}$ restricts to a volume form on $X$, and therefore their pairing is non-trivial: \[\langle[\omega^p],[X]\rangle=\int_X\omega^p>0.\]

We introduce terminology for Poisson manifolds having the analogous property:

\begin{definition}
A Poisson manifold $(M,\pi)$ is said to have {\bf homologically nontrivial Poisson transversals} (or the {\bf HNPT property}) if the homology class \[[X] \in H_{\bullet}(M,\ro_M)\] of any of its compact, nonempty Poisson transversals $X$ is nontrivial.
\end{definition}

\subsection{Examples of Poisson manifolds without the HNPT property}\label{sec : examples}

The examples given here are all of 3-dimensional orientable Poisson manifolds which do not have the HNPT property. In dimension three, the symplectic leaves are either 0-dimensional (i.e.\ zeroes of the Poisson structure) or 2-dimensional. Therefore, a nonempty, compact, connected Poisson transversal is the same as a circle which does not pass through the singular locus and meets the 2-dimensional leaves transversally.

\begin{example}[3-dimensional Lie algebras]\label{ex : 3Lie}
The dual vector space of a Lie algebra $(\mathfrak{g},[\cdot,\cdot])$ is canonically a Poisson manifold $(\mathfrak{g}^*,\pi_{\mathfrak{g}})$, with Poisson structure given by \[\pi_{\mathfrak{g},\xi}(X,Y)=\xi\circ [X,Y],\ \ \textrm{for}\ \xi\in \mathfrak{g}^*, \ X,Y\in T^*_{\xi}\mathfrak{g}^*\simeq \mathfrak{g}.\]
The symplectic leaves of $\pi_{\mathfrak{g}}$ are the coadjoint orbits.

An interesting problem is to give a Lie-theoretic characterization of the Lie algebras whose duals admit compact Poisson transversals. Here, we discuss this problem in detail for the duals of the 3-dimensional Lie algebras. The semisimple ones, $\mathfrak{sl}_2(\R)$ and $\mathfrak{so}_3(\R)$, admit no transverse circles, as one can easily see from their symplectic foliations:
\[
\includegraphics[height=2cm]{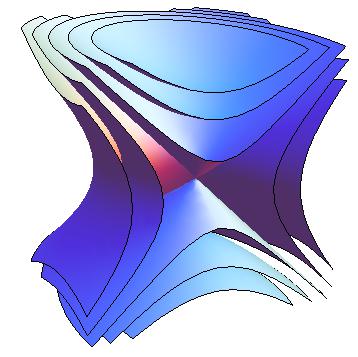} \ \ \ \ \ \ \ \  \ \ \ \ \ \ \ \ \includegraphics[height=2cm]{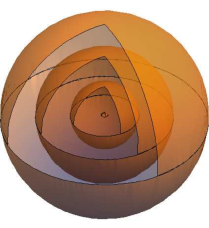}
\]

For the Poisson structures $\pi_{\mathfrak{g}}$ corresponding to the remaining 3-dimensional Lie algebras, there exist coordinates $(x,y,z)$ such that:
\[\pi_{\mathfrak{g}}=X\wedge \frac{\partial}{\partial z}, \]
where $X$ is a linear vector field on $\R^2$, i.e.\
\[X=(ax+by)\frac{\partial}{\partial x}+(cx+dy)\frac{\partial}{\partial y}.\]
The 2-dimensional leaves of the Poisson structure are of the form $C\times \mathbb{R}$, for $C\subset \R^2$ a nontrivial flow line of $X$, and the singular points are of the form $(p,z)$ where $p$ is a zero of $X$.

Poisson structures of this type are easily classified, e.g.\ using the classification of 3-dimensional Lie algebras, or directly: conjugating the matrix $A=\begin{pmatrix}
                    a & b \\
                    c & d \\
                  \end{pmatrix}$
corresponds to linear isomorphisms of $\R^2$, and rescaling the matrix by $t\neq 0$ is equivalent to rescaling the $z$-direction by $1/t$. Following this direct approach, one obtains that $\pi_{\mathfrak{g}}$ (equivalently, $X$) admits a transverse circle if and only if the real part of the two eigenvalues of $A$ have the same sign (and are non-zero); i.e.\ if the flow of $X$ or $-X$ is a contraction towards the origin (this condition is standard in linearization results for vector fields \cite{Sterberg}). Under this condition, all resulting foliations are homeomorphic to each other: the 2-dimensional leaves look like the pages of an open book whose core consists of 0-dimensional leaves. The cases $A=\mathrm{Id}$ and $A$ with eigenvalues that are neither real nor purely imaginary are plotted below:
\[\includegraphics[height=3cm]{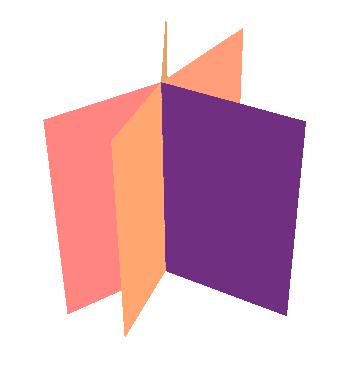}\ \ \ \ \ \ \ \ \includegraphics[height=3cm]{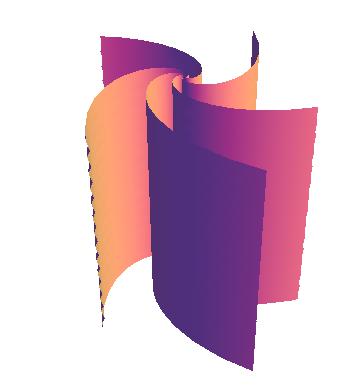}\]
One can easily construct a transverse circle $X$, which goes around the binding of the book. In this case, the homology class $[X]$ of the Poisson transeversal is trivial in $H_1(\mathfrak{g}^*,\R)=0$; in particular, the HNPT property does not hold.
\end{example}

\begin{example}[Reeb foliation of the three-sphere]\label{ex : reeb-poisson}
Consider the Reeb foliation $\mathcal{F}_{\mathrm{R}}$ on the 3-sphere $\mathbb{S}^3$ (see e.g.\ \cite[Section 1.1(5)]{MoerMrc}):
\[
\setlength{\unitlength}{0.8cm}
\begin{picture}(6,2.7)
\put(-1.6,-0.9){\includegraphics[width=3cm,height=3cm]{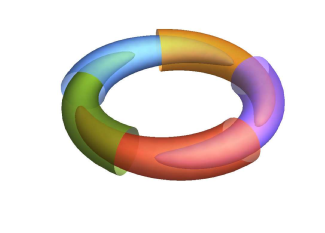} }
\put(6.6,-0.9){\includegraphics[width=3cm,height=3cm]{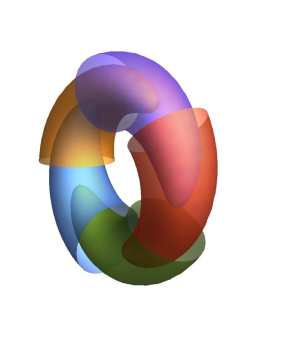} }
\put(-3,1){\Huge $\mathbb{S}^3=$}
\put(2,1){\Huge $+$}
\put(3,-0.3){\includegraphics[width=2.25cm,height=2.25cm]{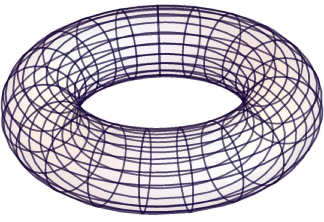} }
\put(6.1,1){\Huge $+$}
\end{picture}
\]
i.e., an unknotted torus $T$ decomposes $\mathbb{S}^3$ into two open solid tori $C_1$ and $C_2$; $T$ is a leaf, and each solid torus is foliated by discs which converge to the boundary in a flat fashion. A choice of volume form on the leaves of $\mathcal{F}_{\mathrm{R}}$ gives a Poisson structure $\pi$ on $\mathbb{S}^3$. The two central circles $X_i\subset C_i$, with $i=0,1$, are Poisson transversals in $(\mathbb{S}^3,\pi)$, and both are homologous to zero, since $H_1(\mathbb{S}^3,\R)=0$. Thus, $(\mathbb{S}^3,\pi)$ does not have the HNPT property.
\end{example}

\subsection{The weak HNPT property}

Even though the examples above do not have the HNPT property, they do satisfy a weak version of it. Namely, in Example \ref{ex : 3Lie} consider the open set $U\subset \R^3$ which is the complement of the binding. Then $U$ is saturated, the transverse circle $X$ is included in $U$, and moreover $[X]\neq 0$ in $H_1(U,\R)$. Moreover, one can easily see that the resulting Poisson manifold $(U,\pi_{\mathfrak{g}}|_{U})$ does have the HNPT property.

To put this example into perspective, we first introduce the following notion:

\begin{definition}
The {\bf saturation} of a Poisson transversal $X\subset (M,\pi)$, denoted by $\mathrm{St}(X)$, is the union of all leaves of $\pi$ which meet $X$.
\end{definition}

Note that $\mathrm{St}(X)$ is open in $M$, simply because transversality of $X$ implies that its image in the leaf space $M/\pi$ is open.

In the discussion of Example \ref{ex : 3Lie} above, we have that $\mathrm{St}(X)=U$, for any Poisson transversal $X$. For the Reeb foliation of Example \ref{ex : reeb-poisson}, we have that $\mathrm{St}(X_i)=C_i$. In fact, these two examples satisfy the following property, which is more natural from a Poisson-geometric point of view:

\begin{definition}
A Poisson manifold $(M,\pi)$ is said to have the {\bf weak HNPT property} if, for any of its compact, nonempty Poisson transversals $X$, the homology class $[X] \in H_{\bullet}(\mathrm{St}(X),\ro_{\mathrm{St}(X)})$ is nontrivial.
\end{definition}

We expect that Poisson manifolds which do not satisfy the weak HNPT property be rather rare; below, we construct such an example.

\begin{example}[Flat bundles over symplectic manifolds]\label{ex : uniformization}
Let $\varrho:M\to (B,\omega)$ be a compact, connected oriented fibre bundle over a symplectic manifold $(B,\omega)$. Assume that $\varrho$ admits a flat Ehresmann connection; in other words, assume that $M$ admits a foliation $\mathcal{F}$ which is complementary to the fibres of $\varrho$. Then there is an induced Poisson structure $\pi$ on $M$, with underlying foliation $\mathcal{F}$ and symplectic forms given by restrictions of ${\varrho}^*(\omega)$ to the leaves. The fibres of $\varrho$ are compact Poisson transversals of $(M,\pi)$. Moreover, the connectedness of $B$ implies that the saturation of each fibre is $M$. On the other hand, note that, up to a non-zero constant, $\varrho^*(\omega^{\mathrm{top}})$ is a representative of the Poincar\'e dual of any fibre of $\varrho$; therefore, the fibres of $\varrho$ are trivial in homology if and only if $\varrho^*(\w^{\mathrm{top}})$ is exact. So, if $\varrho^*(\w^{\mathrm{top}})$ is exact, $(M,\pi)$ does not have the weak HNPT property.

Such examples exist already in dimension 3. Let $\Sigma$ an oriented compact surface of genus $g\geq 2$, and let $\varrho:M\to \Sigma$ be a principal $\mathbb{S}^1$-bundle, with Chern class $c\in H^2(\Sigma,\mathbb{R})$. We will assume that $c$ satisfies:
\[0<|\langle c, [\Sigma]\rangle|\leq 2(g-1).\]
By \cite[Theorem 1.1]{Wood}, the second inequality implies that $\varrho:M\to \Sigma$ admits a transverse foliation $\mathcal{F}$. Let $\omega\in \Omega^2(\Sigma)$ be a symplectic structure on $\Sigma$. Since $c$ can be represented by the curvature of a principal $\mathbb{S}^1$-connection $\alpha$, we have that $\varrho^*(c)=[\dd\alpha]=0$ in $H^2(M,\mathbb{R})$. One the other hand, since $\Sigma$ is 2-dimensional and $c\neq 0$, we have that $[\omega]$ is a multiple of $c$, hence also $\varrho^*([\omega])=0$. We conclude that the Poisson manifold $(M,\pi)$ built out of $\mathcal{F}$ and $\varrho^*(\omega)$ does not have the weak HNPT property. A version of this example is discussed also in \cite[Section 5]{Weinst_mod}.
\end{example}

\begin{remark}
A Poisson manifold $(M,\pi)$ is called {\bf calibrated} if there exists a closed two-form $\w \in \Omega^2(M)$ which restricts on each leaf to the symplectic structure induced by $\pi$. Such Poisson manifolds are necessarily regular (this can be directly seen by using gauge transformations from Dirac geometry). Note that the Poisson structure constructed in Example \ref{ex : uniformization} is calibrated, even by an exact 2-form.

A theorem of D.\ Mart\'{i}nez-Torres \cite[Theorem 2]{DMT13} says that a corank one, compact, calibrated Poisson manifold has compact Poisson transversals of any even codimension, whose saturation is the entire manifold. This can be regarded as the odd-dimensional analogue of Donaldson's result on the existence of symplectic submanifolds \cite{Donaldson}.
\end{remark}

\subsection{Unimodular Poisson manifolds}

In this section we prove that a unimodular Poisson manifold has the HNPT property. Recall that an oriented Poisson manifold $(M,\pi)$ is called {\bf unimodular} \cite{Weinst_mod} if there exists a positive density $\mu\in \Omega^{\mathrm{top}}(M,\ro_M)$ which is invariant under all Hamiltonian vector fields, i.e.\
\[\mathscr{L}_{\pi^{\sharp}(\dd f)}\mu=0, \ \ \textrm{for all} \ f\in C^{\infty}(M).\]

For example, any symplectic manifold $(M,\omega)$ is unimodular. If $o_{\omega}\in \Gamma(\ro_M)$ is the canonical orientation induced by $\omega$, then the invariant positive density is given by
$\omega^{k}\otimes o_{\omega}\in \Omega^{2k}(M,\ro_M)$, where $2k=\mathrm{dim}(M)$.

An invariant density gives rise to several closed $\ro_M$-valued forms:
\begin{lemma}\label{lem : closed}
The positive density $\mu\in \Omega^{\mathrm{top}}(M,\ro_M)$ is invariant under all Hamiltonian vector fields if and only if $\dd\iota_{\pi}\mu=0$. In this case, the following forms are closed:
\[\mu,\ \ \iota_{\pi}\mu,\ \ \iota_{\pi^{2}}\mu,\ \ \ldots\ \iota_{\pi^q}\mu,\ \ldots.\]
\end{lemma}
\begin{proof}
The result is highly standard. The first part appears already in \cite{Weinst_mod}, and follows directly from the relation $\dd f\wedge \dd \iota_{\pi}\mu=\mathscr{L}_{\pi^{\sharp}(\dd f)}\mu$.

For the second part, consider the operator $\mathscr{L}_{\pi}=\iota_{\pi}\dd-\dd\iota_{\pi}$ acting on $\Omega^{\bullet}(M,\ro_M)$ (if $M$ is oriented, this is the usual operator computing Poisson homology). The Poisson condition yields:
\[\mathscr{L}_{\pi}\iota_{\pi^k}-\iota_{\pi^k}\mathscr{L}_{\pi}=\iota_{[\pi,\pi^k]}=0,\]
and explicitly,
\[\iota_{\pi}\dd\iota_{\pi^k}-\dd\iota_{\pi^{k+1}}-\iota_{\pi^{k+1}}\dd+\iota_{\pi^k}\dd\iota_{\pi}=0.\]
Applying this to $\mu$, we obtain that
\[\iota_{\pi}\dd\iota_{\pi^k}\mu=\dd\iota_{\pi^{k+1}}\mu.\]
Thus, the second part follows inductively.

There is also a Dirac geometric approach to the second part: unimodularity is equivalent to the existence of a closed, nowhere-vanishing section of the spinor bundle; in the Poisson case this means the existence of a volume form $\mu$ such that $e^{-\iota_{\pi}}\mu=\mu-\iota_{\pi}\mu+\frac{1}{2}\iota_{\pi^2}\mu-\ldots$ is closed (see \cite{Marco} and section \ref{sec : HNPT Dirac} for details).
\end{proof}

The lemma implies:

\begin{theorem}\label{thm : main}
A unimodular Poisson manifold has the HNPT property.
\end{theorem}
\begin{proof}
Let $X$ be a compact Poisson transversal of codimension $2q$ in a unimodular Poisson manifold $(M,\pi)$ of dimension $m$, and let $\mu$ be an invariant, positive density on $M$. By Lemma \ref{lem : closed} $\iota_{\pi^q}\mu$ is closed. Since ${\pi^q}$ induces a nowhere-vanishing section of the normal bundle of $X$, which induces the coorientation of $X$ (see section \ref{sec : The homology class}), we have that $\iota_{\pi^q}\mu$ restricts to a positive density on $X$. Thus,
\[\langle[\iota_{\pi^q}\mu],[X]\rangle=\int_{X}\iota_{\pi^q}\mu|_X>0,\]
and so $[X]\neq 0$ in $H_{m-2q}(M,\ro_M)$.
\end{proof}

\begin{example}[Lie Algebras]
The canonical Poisson structure $(\mathfrak{g}^*,\pi_{\mathfrak{g}})$ on the dual of a Lie algebra $\mathfrak{g}$, is unimodular if and only if $\mathfrak{g}$ is unimodular as a Lie algebra, in the sense that $\wedge^{\mathrm{top}}\mathfrak{g}$ is trivial as a representation \cite[Section 4]{Weinst_mod}.

For the 3-dimensional Lie algebras of Example \ref{ex : 3Lie}, the two semisimple Lie algebras are unimodular, whereas the Lie algebras given by a matrix $A$ are unimodular if and only if $\mathrm{tr}(A)=0$. Hence, in these cases, Theorem \ref{thm : main} excludes the existence of compact Poisson transversals.
%
\end{example}

\begin{example}[Regular Poisson structures]\label{ex : reg}
For a regular Poisson manifold $(M,\pi)$, unimodularity is equivalent to unimodularity of the underlying foliation $\mathcal{F}$, in the sense that there exists a closed, nowhere-vanishing form $\xi\in \Omega^{q}(M,\ro_M)$, where $q=\mathrm{codim}(\mathcal{F})$ such that $\mathcal{F}$ is given by kernel of $\xi$ \cite{Weinst_mod}.

By Theorem \ref{thm : main}, the foliations in Examples \ref{ex : reeb-poisson} and \ref{ex : uniformization} are not unimodular, because they have compact Poisson transversals with trivial homology class. That these examples are not unimodular was discussed also in \cite{Weinst_mod}.
\end{example}

\subsection{The HNPT property under proper Poisson maps}\label{sec : HNPT under maps}

As shown in \cite{PT1}, Poisson transversals behave well under Poisson maps. Using this, we prove that the HNPT property behaves reasonably well under proper Poisson maps, a result which will be used in the following sections.

\begin{theorem}\label{thm : closed and compact PTs and proper Poisson maps}
Let $f:(P,\pi_P)\to (M,\pi_M)$ be a proper Poisson map. If $(P,\pi_P)$ has the HNPT property, then the homology class of every compact Poisson transversal $X\subset M$ which meets $f(P)$ is nontrivial.

In particular, if $f$ is onto then $(M,\pi_M)$ has the HNPT property.
\end{theorem}

\begin{proof}
Let $X\subset M$ be a compact Poisson transversal satisfying $f(P)\cap X\neq \varnothing$. Let $p=\mathrm{dim}(P)$, $m=\mathrm{dim}(M)$ and $2q=\mathrm{codim}(X)$. By \cite[Lemma 7]{PT1}, $f$ is transverse to $X$ and $Y:=f^{-1}(X)$ is a Poisson transversal in $(P,\pi_P)$. Our assumptions imply that $Y$ is compact and nonempty, and so, since $(P,\pi_P)$ has the HNPT property, we have that the homology class $[Y]$ is nontrivial in $H_{p-2q}(P,\ro_P)$. This is equivalent to the non-triviality of its compactly supported Poincar\'e dual
\[\PD[Y]\in H^{2q}_{c}(P,\R)\] (see Appendix A for Poincar\'e duals in our setting). The conclusion follows because
\begin{equation}\label{eq : PD_funct}
f^*(\PD[X])=\PD[f^{-1}(X)]=\PD[Y]\neq 0 \ \ \textrm{in}\ \ \ H^{2q}_c(P,\R),
\end{equation}
and so $\PD[X]\neq 0$ in $H^{2q}_c(M,\R)$, which is equivalent to $[X]\neq 0$ in $H_{m-2q}(M,\ro_M)$. Identity (\ref{eq : PD_funct}) is a general property of the preimage of a compact submanifold via a smooth, proper map to which it is transverse, and fits into the classical theory of \emph{Umkehr maps} in (co-)homology. For a direct proof, adapt the arguments in \cite[Chapter 1 \S 6]{BottTu82} (see also Appendix A): transversality ensures the existence of tubular neighborhoods $NY\hookrightarrow P$ and $NX\hookrightarrow M$ of $Y$ and $X$, respectively, in which $f$ becomes a fibrewise linear isomorphism preserving the fibrewise orientations (this follows from the Poisson condition); $\PD[X]$ has a representative $\eta_X\in \Omega^{2q}_c(NX)$ whose integral along the fibres of $NX$ is one; the above imply that also $f^*(\eta_X)\in \Omega^{2q}_c(NY)$ has this property, and is therefore a representative of $\PD[Y]$.
\end{proof}

Since symplectic manifolds are unimodular, we obtain:

\begin{corollary}\label{cor : proper symplectic realizations}
A Poisson manifold which admits a surjective proper symplectic realization\footnote{Symplectic realizations are not assumed to be submersive} has the HNPT property.
\end{corollary}

This yields a criterion for the nonexistence of proper symplectic realizations:

\begin{corollary}\label{cor : no H 1}
A regular, corank-one Poisson structure on a compact, oriented manifold $M$ with $H_1(M,\mathbb{R})=0$ does not admit proper symplectic realizations.
\end{corollary}

\begin{proof}
A coorientable, codimension-one foliation on a compact manifold always admits a transverse circle (see e.g.\ \cite{MoerMrc}). In our case, the circle is a compact Poisson transversal, and since $H_1(M,\mathbb{R})=0$, the HNPT property does not hold. Thus, Corollary \ref{cor : proper symplectic realizations} implies the result.
\end{proof}

A second proof of this corollary, suggested by R.\ L.\ Fernandes in the case when the map is also submersive, observes that by \cite[Proposition 7.8]{PMCT1} the existence of a proper submersive symplectic realization implies unimodularity, which in this case (see Example \ref{ex : reg}) implies that the foliation is given by a nowhere-vanishing closed one-form, which cannot exist on a compact manifold $M$ with $H_1(M,\R)=0$. In fact, the following more general property holds: 

\begin{proposition}\label{pro : unimodular poisson}
Let $f : (P,\pi_P) \to (M,\pi_M)$ be a surjective, proper Poisson submersion. If $(P,\pi_P)$ is unimodular, then also $(M,\pi_M)$ is unimodular.
\end{proposition}
\begin{proof}
We will use the map of integration along the fibres
\[f_{\%}:\Omega^{\bullet}(P,\ro_P)\rmap \Omega^{\bullet-p+m}(M,\ro_M),\]
where $p=\mathrm{dim}(P)$ and $m=\mathrm{dim}(M)$. This is defined as follows. First, note that there is a canonical isomorphism $\ro_M\simeq \ro_{f}\otimes f^*(\ro_P)$, where $\ro_f$ is the orientation bundle of $\ker (f_*)$. Let $\omega\in \Omega^{\bullet}(P,\ro_P)$. At any $x\in M$, there is a finite sum decomposition
\[\omega|_{f^{-1}(x)}=\sum_i \omega^i_{f^{-1}(x)}\otimes \omega^i_{M,x},\]
where $\omega^i_{f^{-1}(x)}\in \Omega^{\bullet}(f^{-1}(x),\ro_{f^{-1}(x)})$ and $\omega^i_{M,x}\in \wedge^{\bullet}T_x^*M\otimes \ro_{M,x}$. Then
\[f_{\%}(\omega)_x=\sum_i\Big(\int_{f^{-1}(x)}\omega^i_{f^{-1}(x)}\Big)\omega^i_{M,x}\in \wedge^{\bullet}T_x^*M\otimes \ro_{M,x}.\]

Note also the following general property of fibre integration. For a vector field $\widetilde{v}$ on $P$ which is $f$-related to a vector field $v$ on $M$, we have that $\mathscr{L}_{\widetilde{v}}f_{\%}(\omega)=f_{\%}(\mathscr{L}_{v}\omega)$.

Let $\mu \in \Omega^{p}(P,\ro_P)$ be an positive, invariant density on $(P,\pi_P)$. Note that $f_{\%}(\mu)\in \Omega^{m}(M,\ro_M)$ is a positive density on $M$. By applying the above property to the Hamiltonian vector fields $\widetilde{v}=\pi_P^{\sharp}(\dd f^*(a))$ and $v=\pi_M^{\sharp}(\dd a)$, for $a\in C^{\infty}(M)$, we obtain that $f_{\%}(\mu)$ is invariant under the Hamiltonian vector fields of $\pi_M$. Hence, $(M,\pi_M)$ is unimodular.
\end{proof}

\subsection{Poisson manifolds with closed leaves}\label{sec : closed leaves}

In this section, we prove that Poisson manifolds with closed leaves have the HNPT property. An important class of such Poisson manifolds, the so-called \emph{Poisson manifolds of compact/proper type}, are currently the subject of much investigation \cite{PMCT1,PMCT2,DMT14}.

Let us note the following direct consequence of Theorems \ref{thm : main} and \ref{thm : closed and compact PTs and proper Poisson maps}:

\begin{corollary}\label{cor : unimodular Poisson submanifold}
Let $X$ be a compact Poisson transversal in a Poisson manifold $(M,\pi)$. If $X$ meets a closed, embedded, unimodular Poisson submanifold, then $[X]\neq 0$ in $H_{\bullet}(M,\ro_M)$.

In particular, if $X$ meets a closed symplectic leaf, then $[X]\neq 0$ in $H_{\bullet}(M,\ro_M)$.
\end{corollary}

Closed symplectic leaves are automatically embedded submanifolds, this is the reason why being embedded is not a hypothesis in the second part of the corollary. The proof of this fact is not easily found in the literature, therefore we have included one in Appendix B.

Corollary \ref{cor : unimodular Poisson submanifold} implies the following:

\begin{theorem}\label{thm : all leaves closed implies HNPT}
A Poisson manifold with closed leaves has the HNPT property.
\end{theorem}

\begin{example}[Lie-Poisson spheres]
Let $\mathfrak{g}$ be a Lie algebra of compact type, and endow $\mathfrak{g}^{\ast}$ with an invariant inner product. The corresponding unit sphere becomes a Poisson submanifold of $(\mathfrak{g}^*,\pi_{\mathfrak{g}})$, denoted $(\mathbb{S}(\mathfrak{g}^{\ast}),\pi_{\mathbb{S}})$, and is called the {\bf Lie-Poisson sphere} corresponding to $\mathfrak{g}$. Lie-Poisson spheres are interesting examples of compact Poisson manifolds: they are not integrable (with a few exceptions); they have compact leaves; and, in the semi-simple case, their local deformation space (which is infinite dimensional) can be described explicitly \cite{Mar_def}.

Theorem \ref{thm : all leaves closed implies HNPT} implies that Lie-Poisson spheres have the HNPT property. Note that, for a compact Poisson transversal $X\subset\mathbb{S}(\mathfrak{g}^{\ast})$, it can happen that $[X]\neq 0$ only when $X=\mathbb{S}(\mathfrak{g}^{\ast})$, or when $X$ is a finite set of points, in which case $\mathbb{S}(\mathfrak{g}^{\ast})$ is symplectic, and so $\mathfrak{g}\simeq\mathfrak{so}(3)$.

In fact, it can be easily seen that the Lie-Poisson spheres are unimodular: if $\mu\in \wedge^{\mathrm{top}}\mathfrak{g}$ is non-zero, then $\iota_{\mathcal{E}}\mu|_{\mathbb{S}(\mathfrak{g}^{\ast})}$ is an invariant volume form, where $\mathcal{E}$ denotes the Euler vector field on $\mathfrak{g}^*$.
\end{example}

\subsection{Log-symplectic structures}\label{sec : log-symplectic}

In this section we discuss the homology of compact Poisson transversals in log-symplectic manifolds. Recall \cite{GMP} that a Poisson structure $\pi$ on a manifold $M$ of even dimension $2k$ is called {\bf log-symplectic} (or {\bf $b$-symplectic}) if its top power $\pi^k\in \Gamma(\wedge^{2k}TM)$ is transverse to the zero-section of $\wedge^{2k}TM$. The \textbf{singular locus} of the log-symplectic structure $\pi$ is the codimension-one submanifold where the rank is not maximal
\[Z(\pi)=\{x\in M \ | \ \pi^k_x=0\}.\]

Log-symplectic manifolds are not unimodular, unless they are symplectic \cite{GMP}.

We have that:
\begin{theorem}\label{thm : log weak HNPT}
Log-symplectic manifolds have the weak HNPT property.
\end{theorem}
\begin{proof}
Let $X$ be a nonempty, compact Poisson transversal in a log-symplectic manifold $(M,\pi)$. First, assume that $X$ meets the singular locus $Z(\pi)$. We have that $Z(\pi)$ is a closed, embedded, unimodular \cite{GMP} Poisson submanifold; therefore, by Corollary \ref{cor : unimodular Poisson submanifold}, $[X]\neq 0$ in $H_{\bullet}(M,\ro_M)$, and so also
in $H_{\bullet}(\mathrm{St}(X),\ro_{\mathrm{St}(X)})$. On the other hand, if $X\cap Z(\pi)=\varnothing$, then the saturation of $X$ is a union of connected components of the symplectic manifold $(M\backslash Z(\pi),\pi^{-1})$. Thus, $X$ is a symplectic submanifold of $\mathrm{St}(X)$, and so $[X]\neq 0$ in $H_{\bullet}(\mathrm{St}(X),\ro_{\mathrm{St}(X)})$.
\end{proof}

\begin{example}\label{ex : log are not HNPT 1}
Log-symplectic manifolds do not satisfy the HNPT property in general. For example, consider on the unit sphere $\mathbb{S}^2\subset \R^3$ the log-symplectic structure given in cylindrical coordinates $(r=1,\theta,z)$ by $\pi_{\mathbb{S}^2}=z\frac{\partial}{\partial z}\wedge \frac{\partial}{\partial \theta}$, and let $X=\{N,S\}$ consist of two points, $N$ on the northern hemisphere and $S$ on the southern hemisphere. Since the induced coorientations at $N$ and $S$ differ, we have that $[X]=[N]-[S]=0$. This example can be generalized to any orientable log-symplectic manifold which is not symplectic, by choosing two point on different sides of the singular locus.
\end{example}
\begin{example}\label{ex : log are not HNPT 2}
Note that $\pi_{\mathbb{S}^2}$ o Example \ref{ex : log are not HNPT 1} is invariant under the antipodal action of $\mathbb{Z}_2$, and that it descends to a log-symplectic structure $\pi_{\mathbb{P}^2}$ on the projective plane  $\mathbb{P}^2=\mathbb{S}^2/\mathbb{Z}_2$. In this case, for any point $P$ in the symplectic locus of $\pi_{\mathbb{P}^2}$, we have that $X=\{P\}$ is a Poisson transversal, but $[X]=0$, because $H_0(\mathbb{P}^2,\ro_{\mathbb{P}^2})=0$. This example can be generalized to any non-orientable log-symplectic manifold.
\end{example}

Nevertheless, the only issues that prevent a log-symplectic manifold from having the HNPT property are those discussed in Examples \ref{ex : log are not HNPT 1} and \ref{ex : log are not HNPT 2} above:

\begin{theorem}
A compact, connected, nonempty Poisson transversal of an orientable log-symplectic manifold has nontrivial homology class.
\end{theorem}

\begin{proof}
If the Poisson transversal meets the singular locus, then the connectivity assumption is not needed, and the argument from the proof of Theorem \ref{thm : log weak HNPT} applies.

Assume that the Poisson transversal $X$ does not intersect the singular locus $Z(\pi)$. Consider the closed 2-form $\omega:=\pi^{-1}$, which is singular at $Z(\pi)$. Note that $X$ is orientable, because $M$ is orientable and $X$ is coorientable, and fix an orientation on $X$. Since $X$ is connected, $\int_X\omega^k\neq 0$, where $2k=\mathrm{dim}(X)$. Let $U\simeq (-1,1)\times Z(\pi)$ be a tubular neighborhood of $Z(\pi)$ in $M$ which does not meet $X$. It is shown in \cite{Mar_OT} that there exists a closed 2-form $\omega'$ on $M$ which coincides with $\omega$ outside of $U$ (note that the argument from \emph{loc.cit.}\ does not use compactness of $Z(\pi)$, but only that $Z(\pi)$ is closed and embedded). We obtain that
\[\langle[\omega']^k,[X]\rangle=\int_X(\omega')^k=\int_X\omega^k\neq 0,\]
which implies the conclusion.
\end{proof}

\subsection{The general Dirac case}\label{sec : HNPT Dirac}

The results presented in this paper have a natural and rather straightforward  generalization to the setting of Poisson transversals in Dirac manifolds, and we devote this section to explaining this. For the basics of Dirac geometry, we recommend \cite{ABM,H,Marco}.

Let $L\subset TM\oplus T^*M$ be a Dirac structure on a manifold $M$. Let
\[K_L\subset \wedge^{\bullet}T^*M\]
denote the \textbf{spinorial line bundle} corresponding to $L$. The \textbf{orientation double cover} of $K_L$ will be denoted by
\[\widetilde{M}_L:=(\mathrm{K}_L \diagdown M)/\mathbb{R}_{>0},\]
and the \textbf{orientation bundle} $K_L$ by
\[\ro_L:=(\widetilde{M}_L\times\R)/\mathbb{Z}_2.\]
Note that $\ro_L$ has a canonical flat connection, and a canonical metric. For Poisson structures, $\ro_L$ is isomorphic to the orientation bundle $\ro_M$ of $M$. The entire discussion in Appendix A can be directly adapted to this setting: one obtain the complex $(\Omega^{\bullet}(M,\ro_L),\dd)$ computing the cohomology $H^{\bullet}(M,\ro_L)$ of $M$ with coefficients in $\ro_L$; one can define the homology $H_{\bullet}(M,\ro_L)$ of $M$ with local coefficients in $\ro_L$; using the canonical metric on $\ro_L$, there is an induced integral pairing:
\begin{equation}\label{eq : pairing}
H^{\bullet}(M,\ro_L)\times H_{\bullet}(M,\ro_L)\rmap \R,
\end{equation}
which gives an isomorphism $H^{\bullet}(M,\ro_L)\simeq H_{\bullet}(M,\ro_L)^*$.\\

Next, we recall \cite{ABM,H} the various types of maps $f:(P,L_P)\to (M,L_M)$ between Dirac manifolds. For
\[a=u+\alpha\in T_xP\oplus T_x^*P\ \ \textrm{and}\ \  b=v+\beta\in T_{f(x)}M\oplus T_{f(x)}^*M,\]
we write $a\sim_f b$ if $v=f_*(u)$ and $\alpha=f^*(\beta)$. The map $f$ is called
\begin{itemize}
\item \textbf{backward Dirac} if, for all $x\in P$,
 \[L_{P,x}=f^{!}(L_{M,f(x)}):=\{ a : \exists\ b\in L_{M,f(x)} \ \textrm{s.t.}\ a\sim_fb\};\]
\item \textbf{transverse} if $f$ is transverse to the leaves of $L_M$, i.e.\ $\ker(f^*)\cap L_M=0$.
\end{itemize}

If $f$ is backward Dirac, then $L_P$ is determined by $L_M$, and we write $L_P=f^{!}(L_M)$.

The condition that $f$ be transverse to the leaves of $L_M$ implies that $f^{!}(L_M)$ is a smooth Dirac structure, and that $f:(P,f^{!}(L_M))\to (M,L_M)$ is a backward Dirac map \cite[Proposition 5.6]{H}. The transversality condition is equivalent to
\[f^*(K_{L_M})\subset \wedge^{\bullet}T^*P\]
being a line bundle, and in this case, $f^*(K_{L_M})=K_{L_P}$, \cite[Lemma 1.9]{ABM}.

Dually, we have the following notions. A map $f:(P,L_P)\to (M,L_M)$ is called
\begin{itemize}
\item \textbf{forward Dirac} if, for all $x\in P$,
\[L_{M,f(x)}=f_{!}(L_{P,x}):=\{ b : \exists\ a\in L_{P,x}\ \textrm{s.t.}\ a\sim_fb\};\]
\item \textbf{strong} if $\ker(f_*)\cap L_P=0$.
\end{itemize}

If $f$ is forward Dirac, then $L_P$ determines $L_M$ along the image of $f$.

Pointwise, the strong forward condition is dual to the transverse backward condition. Therefore, it can be characterized in terms of a dual notion, which we now discuss (see \cite{Eckhard_book}). The \textbf{co-spinorial line bundle} of a Dirac manifold $(N,L)$ is the real line bundle
\[C_L\subset \wedge^{\bullet}TN\]
consisting of elements $w\in \wedge^{\bullet}TN$ such that
\[u\wedge w+\iota_{\xi}w=0, \ \ \textrm{for all}\ u+\xi \in L.\]
There exists a natural isomorphism between the following bundles (see \cite[Sections 3.2-3.4, 4.1]{Eckhard_book})
\begin{equation}\label{eq : cospinor}
K_L\otimes \wedge^{\mathrm{top}}TN\diffto C_L,\ \ \ \varphi\otimes w\mapsto \iota_{\varphi}w.
\end{equation}

By duality, we have that $f:(P,L_P)\to (M,L_M)$ is a strong forward Dirac map if and only if, for every $x\in P$,
\begin{equation}\label{eq : cospinor under strong}
f_*(C_{L_P,x})=C_{L_M,f(x)}.
\end{equation}
Using also the canonical isomorphism above, we obtain that a strong forward Dirac map induces the following isomorphisms
\begin{equation}\label{eq : isos}
K_{L_P}\otimes \wedge^{\mathrm{top}}TP\simeq C_{L_P}\stackrel{f_*}{\simeq}f^*(C_{L_M})\simeq f^*(K_{L_M}\otimes \wedge^{\mathrm{top}}TM).
\end{equation}

\begin{definition}
A \textbf{Dirac transversal} in $(M,L)$ is an embedded submanifold $X\subset M$ which intersects the presymplectic leaves of $L$ transversally.

A \textbf{Poisson transversal} in $(M,L)$ is an embedded submanifold $X\subset M$ which intersects the presymplectic leaves of $L$ transversally and symplectically.
\end{definition}

\begin{example}
A Poisson transversal in a foliated manifold $(M,\mathcal{F})$ is an embedded submanifold $X\subset M$ which is of complementary dimension to $\mathcal{F}$ and meets each leaf transversally. A Poisson transversal in a manifold endowed with a closed two-form $(M,\omega)$ is an embedded submanifold $X$ for which $\omega|_X$ is symplectic.
\end{example}

The following result characterizing Poisson transversals is straightforward.
\begin{lemma}\label{lem : when DTs are PTs}
Let $i:X\hookrightarrow (M,L)$ be an embedded submanifold. The following conditions are equivalent:
\begin{enumerate}[(a)]
\item $X$ is a Poisson transversal;
\item $(TX \oplus N^*X)\cap L = 0$;
\item The top degree of $i^*(\mathrm{K}_L)\subset \wedge^{\bullet}T^*X$ vanishes nowhere;
\item $\mathrm{pr}(C_L|_X)=\wedge^{q}NX$, where $q=\mathrm{codim}(X)$ and $\mathrm{pr}:\wedge^{q}TM|_{X}\to \wedge^qNX$ is the natural projection.
\end{enumerate}
\end{lemma}

Let $i:X\hookrightarrow (M,L)$ be a compact Poisson transversal. Since $i$ is transverse to the leaves, as discussed above, the pullback $i^!(L)$ is a Dirac structure on $X$ with spinorial line $K_{i^{!}(L)}=i^*(K_L)$. Since $X$ is a Poisson transversal, by Lemma \ref{lem : when DTs are PTs} (c), we have that $i^{!}(L)$ corresponds to a Poisson structure $\pi_X$ on $X$. Moreover, we obtain isomorphisms of line bundles:
\[i^*:K_L|_{X}\diffto K_{i^{!}(L)},\ \ \mathrm{pr}_{\mathrm{top}}:K_{i^{!}(L)}\diffto \wedge^{\mathrm{top}}T^*X,\]
where the second map is just the projection on the top degree component of $\wedge^{\bullet}T^*X$. The composition induces a flat isomorphisms of the associated orientation bundles
\begin{equation}\label{eq : iso}
\ro_L|_{X}\simeq \ro_{X}.
\end{equation}
Using this isomorphism and functoriality of homology with coefficients \cite[Chapter VI, 2]{Whitehead}, $X$ induces a homology class in \[[X]\in H_{\bullet}(M,\ro_L),\]
obtained by pushing forward the fundamental class of $X$ in $H_{\bullet}(X,\ro_X)$.

The compactly supported Poincar\'e dual of $X$ can be constructed on any tubular neighborhood $NX$ of $X$, as a closed form in $\eta\in \Omega_{c}^q(NX,\ro_{NX})$, where $q=\mathrm{codim}(X)$ and by $\ro_{NX}$ we have denoted the orientation bundle of the vector bundle $\wedge^qNX$. The restriction of $\eta$ to the fibres of $p:NX\to X$ is a compactly supported density, and the cohomology class $[\eta]=\PD[X]\in H_{c}^q(NX,\ro_{NX})$ is determined by the fact that $\eta$ integrates to 1 over these fibers. On the other hand, there are canonical isomorphisms $\ro_{NX}\simeq \ro_{M}|_{NX}\otimes p^*(\ro_X)$,
and, since $X$ is a Poisson transversal, $\ro_L|_{X}\simeq \ro_{X}$ (\ref{eq : iso}); thus, we obtain an isomorphism
\[\ro_{NX}\simeq (\ro_{M}\otimes \ro_{L})|_{NX}.\]
Using this map, we can push forward $\PD[X]$ and regard it as an element in
\[\PD[X]\in H_{c}^{q}(M,\ro_M\otimes \ro_L).\]

\begin{definition}
A Dirac manifold $(M,L)$ has the \textbf{HNPT property} (resp.\ the \textbf{weak HNPT property}) if any of its nonempty compact Poisson transversals $X$ has a non-trivial homology class $[X]$ in $H_{\bullet}(M,\ro_L)$ (resp.\ in $H_{\bullet}(\mathrm{St}(X),\ro_L)$, where $\mathrm{St}(X)$ is the saturation of $X$, i.e.\ the union of all presymplectic leaves that meet $X$).
\end{definition}

Let us now discuss the Dirac version of Theorem \ref{thm : main}. The line bundle $K_L\otimes \ro_L\subset \wedge^{\bullet}T^*M\otimes \ro_L$ is trivializable, and moreover, it carries a canonical orientation: for $\xi \in K_L$, with $\xi\neq 0$, the element $\xi\otimes [(\R_{>0}\cdot\xi,1)]$ is positive and the element $\xi\otimes [(\R_{>0}\cdot\xi,-1)]$ is negative. The Dirac manifold $(M,L)$ is called \textbf{unimodular} if $K_L\otimes \ro_L$ admits a closed, nowhere-vanishing section \cite{ELW,Marco}:
\[\mu\in \Gamma(K_L\otimes \ro_L)\subset \Omega^{\bullet}(M,\ro_L).\]

\begin{thm}
 Unimodular Dirac manifolds have the HNPT property.
\end{thm}
\begin{proof}
 Choose such a closed section $\mu$ which is positive. Let $i:X\hookrightarrow(M,L)$ be a  nonempty Poisson transversal. Using Lemma \ref{lem : when DTs are PTs}, and the canonical isomorphism $\ro_L|_X\simeq \ro_X$, we obtain that $\mathrm{pr}_{\mathrm{top}}i^{*}(\mu)\in \Omega^{\mathrm{top}}(X,\ro_X)$ is a positive density on $X$. Therefore, the pairing (\ref{eq : pairing}) gives
\[\langle[\mu],[X]\rangle=\int_{X}\mathrm{pr}_{\mathrm{top}}i^{*}(\mu)>0,\]
and so $[X]\neq 0$ in $H_{\bullet}(M,\ro_L)$.
\end{proof}

Next, we discuss the results of section \ref{sec : HNPT under maps} in the Dirac setting.
The following is a straightforward generalization of \cite[Lemma 7]{PT1}.

\begin{lemma}\label{lem : strong maps and pts}
Let $f : (P,L_P) \to (M,L_M)$ be a forward map between Dirac manifolds. If $X \subset (M,L_M)$ is a Dirac transversal, then $f$ is transverse to $X$, and $Y:=f^{-1}(X)\subset (P,L_P)$ is also a Dirac transversal. If in addition $f$ is strong and $X$ is a Poisson transversal, then $Y$ is also a Poisson transversal.
\end{lemma}

The Dirac version of Theorem \ref{thm : closed and compact PTs and proper Poisson maps} holds for a proper, strong forward Dirac map $f:(P,L_P)\to (M,L_M)$ (instead of a Poisson map):

\begin{thm}\label{thm : closed and compact PTs and proper strong forward maps}
Let $f:(P,L_P)\to (M,L_M)$ be a proper, strong, forward map. If $(P,L_P)$ has the HNPT property, then the homology class of every compact Poisson transversal $X\subset M$ which meets $f(P)$ is nontrivial.
\end{thm}
\begin{proof}
Let $X$ be a compact Poisson transversal of $L_M$ of codimension $q$. Then by Lemma \ref{lem : strong maps and pts}, $Y:=f^{-1}(X)$ is a compact Poisson transversal of $L_P$. As discussed above, the compactly supported Poincar\'e duals of these Poisson transversals are elements of:
\[\PD[X]\in H_c^{q}(M,\ro_M\otimes \ro_{L_M}),\ \ \PD[Y]\in H_c^{q}(P,\ro_P\otimes \ro_{L_P}).\]
By the isomorphisms (\ref{eq : isos}), the strong forward Dirac map $f$ induces an isomorphism (of flat bundles endowed with metrics):
\[\ro_P\otimes \ro_{L_P}\simeq f^*(\ro_M\otimes \ro_{L_M}),\]
so we have a well-defined map
\[f^*:H_c^{q}(M,\ro_M\otimes \ro_{L_M})\rmap H_c^{q}(P,\ro_P\otimes \ro_{L_P}).\]
The argument from the proof of Theorem \ref{thm : closed and compact PTs and proper Poisson maps} is easily adapted to conclude that
\[f^*(\PD[X])=\PD[Y].\]
So if $[Y]\neq 0$ then also $[X]\neq 0$.
\end{proof}

Since closed 2-forms are unimodular as Dirac structures, the Dirac versions of Corollary \ref{cor : proper symplectic realizations} holds if one replaces symplectic realizations by \emph{presymplectic realizations} \cite{BCWZ}, i.e. strong, forward Dirac maps $f:(P,\omega)\to (M,L)$, where $\omega$ is a closed 2-form.

\begin{cor}
A Dirac manifold which admits a surjective proper presymplectic realization has the HNPT property.
\end{cor}

The Dirac version of Proposition \ref{pro : unimodular poisson} also holds.
\begin{prop}\label{pro : unimodular dirac}
Let $f:(P,L_P)\to (M,L_M)$ be a strong forward Dirac map, which is a proper surjective submersion. If $L_P$ is unimodular, then also $L_M$ is unimodular.
\end{prop}

\begin{proof}
Denote $p:=\mathrm{dim}(P)$, $m:=\mathrm{dim}(M)$ and $q:=p-m$. Let $x\in P$, and consider non-zero elements $\varphi\in K_{L_P,x}$ and $w\in \wedge^{p}T_xP$. By (\ref{eq : cospinor}), $\iota_{\varphi}w$ spans the co-spinorial line $C_{L_P,x}$, and so by (\ref{eq : cospinor under strong}) $f_*(\iota_{\varphi}w)$ spans the co-spinorial line $C_{L_M,f(x)}$. On the other hand, decomposing $w=v\wedge u$, with $v\in \wedge^q\ker_x(f_*)$, we have that $\iota_v\varphi=f^*(\psi)$, for some $\psi\in \wedge^{\bullet}T^*_{f(x)}M$. Since $f_*$ annihilates the components of $v$, we obtain that $f_*(\iota_{\varphi}w)=f_*(\iota_{\iota_v\varphi}u)=\iota_{\psi}f_*(u)$, which, again by (\ref{eq : cospinor}) implies that $\psi$ is a non-zero element of $K_{L_M,f(x)}$. We obtain an isomorphism of line bundles:
\begin{equation}\label{eq : push-forward spinor}
K_{L_P}\otimes \wedge^q\ker(f_*)\simeq f^*(K_{L_M}),\ \ \ \varphi\otimes v\mapsto \iota_{v}\varphi=f^*(\psi).
\end{equation}
This induces a flat metric preserving isomorphism of the orientation bundles
\[\ro_{L_P}\otimes \ro_f\simeq f^*(\ro_{L_M}).\]
Therefore, there is an integration along the fibres map
\[f_{\%}:\Omega^{\bullet}(P,\ro_{L_P})\rmap \Omega^{\bullet-q}(M,\ro_{L_M}),\]
defined as in the proof of Proposition \ref{pro : unimodular poisson}. Moreover, from (\ref{eq : push-forward spinor}), if follows that a positive section $\mu\in \Gamma(K_{L_P}\otimes\ro_{L_P})$ is sent to a positive section $f_{\%}(\mu)\in\Gamma(K_{L_M}\otimes\ro_{L_M})$. Finally, the conclusion follows because $f_{\%}$ is also a chain map (which is shown as in the classical case), so $f_{\%}(\mu)$ is also closed, if $\mu$ is closed.
\end{proof}

The above results imply that Corollary \ref{cor : unimodular Poisson submanifold} and Theorem \ref{thm : all leaves closed implies HNPT} extend directly to the Dirac setting:

\begin{cor3}
If a compact Poisson transversal meets a closed, embedded, unimodular Dirac submanifold (in particular, if it meets a closed presymplectic leaf), then its homology class is non-trivial.
\end{cor3}

\begin{thm}
 A Dirac manifold with closed leaves has the HNPT property.
\end{thm}

The Dirac analogue of log-symplectic structures was recently introduced in \cite[Definition 4.11]{Blohmann} under the name of log-Dirac structures, and we believe that also the result of section \ref{sec : log-symplectic} generalize to this setting. However, the theory of these structures has not been developed yet, and so we do not investigate this here.

\subsection*{Appendix A: (Co-)homology twisted by the orientation bundle}\label{appendix A}
We briefly recall the cohomology with values in the orientation bundle, for details see \cite[Chapter I \S 7]{BottTu82}. The \textbf{orientation bundle} of a manifold $M$ is the real line bundle
\[\ro_M:=(\widetilde{M}\times \R)/\mathbb{Z}_2,\]
where $\widetilde{M}$ is the orientation double cover of $M$. In other words, $\ro_M$ is trivial over each chart, and the transition function between two charts is multiplication by the sign of the Jacobian determinant of the change of chart map. The bundle $\ro_M$ carries a canonical flat connection, giving rise to a differential $\dd$ on the complex of $\ro_M$-valued differential forms $\Omega^{\bullet}(M,\ro_M)$. The resulting cohomology $H^{\bullet}(M,\ro_M)$ is the \textbf{cohomology of $M$ with values in the orientation bundle}. Note that an orientation on $M$ (if it exists) induces an isomorphism of complexes
\[(\Omega^{\bullet}(M,\ro_M),\dd)\simeq (\Omega^{\bullet}(M),\dd);\]
thus, one obtains the usual de Rham cohomology of $M$.

Recall also that elements in $\Omega^{\mathrm{top}}(M,\ro_M)$ are called (smooth) densities, and that compactly supported densities can be canonically integrated over $M$. This ensures the existence of a pairing:
\[\Omega^{\mathrm{top}-k}(M,\ro_M)\times \Omega^{k}_{c}(M)\stackrel{\wedge}{\rmap} \Omega^{\mathrm{top}}_c(M,\ro_M)\stackrel{\int_M}{\rmap}\R,\]
which descends to cohomology, and induces an isomorphism:
\[H^{\mathrm{top}-k}(M,\ro_M)\simeq H^k_c(M,\R)^*.\]

For an exposition of homology with local coefficients see \cite[Chapter VI, 2]{Whitehead}. Let us describe the complex $(C_{\bullet}(M,\ro_M),\partial)$ computing it.
For a continuous map $u:\Delta^q\to M$, where $\Delta^q$ is the standard $q$-simplex, consider the 1-dimensional vector space $V_u$ consisting of all the flat lifts $v:\Delta^q\to \ro_M$ of $u$ (i.e.\ $v(\Delta^q)$ is included in a single leaf of the foliation on $\ro_M$ induced by the flat connection). Then $C_{q}(M,\ro_M)=\oplus_{u\in C^0(\Delta^q,M)} V_{u}$, and $\partial$ is given by a formula similar to the that of the usual differential on the singular chains. The resulting \textbf{homology of $M$ with local coefficients in $\ro_M$} is denoted by $H_{\bullet}(M,\ro_M)$.

Note that if $m=\mathrm{dim}(M)$ and $u:\Delta^m\to M$ is a smooth embedding, then $u(\Delta^m)$ inherits an orientation from the standard orientation on $\Delta^m$. Thus $u$ has a canonical flat lift $u^c\in V_u$, $u^c:\Delta^m\to \widetilde{M}\subset \ro_M$.

Assume now that $M$ is compact, and consider a smooth triangulation of $M$ into embedded $m$-simplices $u_i:\Delta^m\to M$, $i=1,\ldots,n$. Define the \textbf{fundamental class of the compact manifold $M$} as
\[[M]:=\Big[\sum_{i=1}^n u^c_i\Big]\in H_m(M,\ro_M).\]
Standard arguments show that the sum $\sum_{i=1}^nu^c_i$ is indeed closed, and that the class it defines is independent of the triangulation. Note that, if $M$ is orientable, an orientation $\gamma$ induces an isomorphism $H_{\bullet}(M,\ro_M)\simeq H_{\bullet}(M,\R)$, under which $[M]\in H_m(M,\ro_M)$ corresponds to the fundamental class of the oriented manifold $[(M,\gamma)]\in H_m(M,\R)$. However, by working with local coefficients in $\ro_M$, the fundamental class is independent of the choice of orientation.

Any \textbf{compact cooriented submanifold} $X\subset M$ has a fundamental class
\[[X]\in H_{\mathrm{dim}(X)}(M,\ro_M).\]
This is defined by pushing forward the fundamental class of $X$, by using functoriality of homology with local coefficients (see \emph{loc.cit.}), and by using that a coorientation induces an isomorphism $\ro_X\simeq \ro_M|_X$.

Let us remark that also the de Rham theorem holds for (co-)homology twisted by the orientation bundle. First, note that $\ro_M$ comes with a fibre metric $\langle\cdot,\cdot \rangle_{\ro_M}$, defined so that elements of the form $[p,\pm 1]$ have length one. This can be used to induce an integral pairing:
\[(\eta, v)\in \Omega^q(M,\ro_M)\times C_q(M,\ro_M)\rmap \langle v, u^*\eta\rangle_{\ro_M}\in \Omega^q(\Delta^q)\stackrel{\int_{\Delta^q}}{\rmap} \R,\]
where $u:\Delta^q\to M$ is the projection of $v$. This pairing descends to a pairing \[\langle\cdot,\cdot\rangle: H^q(M,\ro_M)\times H_q(M,\ro_M)\rmap \R,\]
which is non-degenerate and induces a de Rham isomorphism:
\[H^q(M,\ro_M)\simeq H_q(M,\ro_M)^*.\]
For oriented manifolds, this is just the usual de Rham theorem. For non-orientable ones, this can also be deduced from the usual de Rham theorem by working on the oriented double cover $\widetilde{M}$ and identifying $H^q(M,\ro_M)$ (resp.\ $H_q(M,\ro_M)$) with the $-1$ eigenspace of the nontrivial deck transformation on $H^{q}(\widetilde{M},\R)$ (resp.\ $H_q(\widetilde{M},\R)$).

Note also that the fundamental class $[M]$ of a compact manifold $M$ is non-trivial in $H_{\bullet}(M,\ro_M)$; this can be easily seen by pairing with a positive density.

Let $X\subset M$ be a compact, cooriented submanifold of codimension $q$. Then $X$ has a compactly supported Poincar\'e dual $\PD[X]\in H^q_c(M,\R)$. A compactly supported closed form $\eta_X\in \Omega^q_c(M)$ is a representative of $\PD[X]$ if and only if, for any closed $(m-q)$-form $\omega\in \Omega^{m-q}(M,\ro_M)$, we have that:
\[\int_M\omega\wedge \eta_X=\int_X\omega|_X,\]
where $\omega|_X$ is viewed as an element of $\Omega^{m-q}(M,\ro_X)$ by using the isomorphism $\ro_M|_X\simeq \ro_X$ induced by the coorientation. In other words, $[X]$ and $\PD[X]$ give the same result when paired with elements in $H^{m-q}(M,\ro_M)$. In fact, $\eta_X$ can be chosen to be supported inside the image of any tubular neighborhood $NX\hookrightarrow M$ of $X$ in $M$ (we are considering tubular neighborhoods which, along $X$ induce the identity on normal bundles). The class $[\eta_X]\in H_c^{q}(NX,\R)$ is characterized by the property that it integrates to $1$ over the fibers of $NX$ (which are oriented!). For these claims, adapt the arguments in \cite[Chapter 1 \S 5, \S 6, \S 7]{BottTu82}.

\subsection*{Appendix B: Closed symplectic leaves}\label{appendix B}

In this section, we prove that closed leaves of Poisson manifolds are embedded. This result was used to deduce Corollary \ref{cor : unimodular Poisson submanifold} and Theorem \ref{thm : all leaves closed implies HNPT}. The result holds in more generality (see below); for foliations it seems well-known (e.g.\ \cite[Corollary 7.4]{Sharpe}), and for singular foliations it is stated in \cite[Proposition 2.2]{Dazord} and \cite[Lemma 1.4.5]{Ratiu}, however without a proof.

\begin{proposition}\label{prop : closed embedded}
A closed leaf of a Poisson manifold is an embedded submanifold.
\end{proposition}
\begin{proof}
Let $S$ be a closed symplectic leaf of a Poisson manifold $(M,\pi)$. By Weinstein's Splitting Theorem \cite[Theorem 2.1]{Wein}, for every $p\in S$ there exist an open neighborhood $U\subset M$ and a diffeomorphism $\varphi:U\diffto V\times W$, such that:
\begin{itemize}
\item $V$ is a connected open neighborhood of $p$ in $S$;
\item $W$ is an embedded submanifold of $M$ through $p$ which is transverse to $S$;
\item $\varphi|_V=\mathrm{id}_V$ and $\varphi|_W=\mathrm{id}_W$;
\item for any leaf $S'$, $\varphi(S'\cap U)=V\times \Lambda_{S'}$, where $\Lambda_{S'}\subset W$.
\end{itemize}

Denote by $\mathrm{Diff}_S(M)$ the group of diffeomorphisms of $M$ that send $S$ to itself. Using the above decomposition, for any point $q\in V$, one can build a compactly supported diffeomorphism in $\mathrm{Diff}_S(M)$ which sends $p$ to $q$. Hence, the orbits of $\mathrm{Diff}_S(M)$ on $S$ are open, and since $S$ is connected, we conclude that $\mathrm{Diff}_S(M)$ acts transitively on $S$.

We claim that $S$ being closed implies that $p$ is isolated in $\Lambda_S$. From this the conclusion follows: by shrinking $W$, we may assume that $\Lambda_{S}=\{p\}$. Since $p$ is arbitrary, we deduce that the open sets $U\cap S$, with $U$ open in $M$, form a basis for the topology of $S$; hence $S$ is embedded.

To prove the claim, let $q\in \Lambda_S$ and let $\chi\in \mathrm{Diff}_S(M)$ be a diffeomorphism sending $p$ to $q$. The map $\varphi\circ \chi$ is a diffeomorphism between a neighborhood of $p$ in $U$ and a neighborhood of $(p,q)$ in $V\times W$. Moreover, since $\chi$ preserves $S$, it sends the plaque $V$ into the plaque $V\times \{q\}$. This implies that $\chi$ induces a diffeomorphism transverse to these plaques:
\[\theta=\mathrm{pr}_W\circ \varphi\circ \chi|_{W}:W_p\diffto W_q,\]
where $W_p$ and $W_q$ are open neighborhoods in $W$ of $p$ and $q$, respectively. Again, since $\chi$ preserves $S$, $\theta(\Lambda_S\cap W_p)\subset \Lambda_S\cap W_q$. This implies that, if $p$ is an accumulation point of $\Lambda_S$, then so is $q$. Thus, assuming that $p$ is not isolated in $\Lambda_S$, we obtain that every point in $\Lambda_S$ is an accumulation point. On the other hand, since $S$ is second countable, $\Lambda_S$ is at most countable, and since $S$ is closed and $W$ is embedded, $\Lambda_S$ is closed in $W$. This implies that $\Lambda_S$ is a countable perfect set in $W$, which contradicts \cite[Theorem 2.43]{Rudin}; equivalently, Baire's Category Theorem is contradicted for the closed set $\Lambda_S$, because $\varnothing=\cap_{q\in \Lambda_S}(\Lambda_S\backslash\{q\})$. This concludes the proof.
\end{proof}

Inspired by the argument above, we introduce the following class of submanifolds.
\begin{definition}
A subset $S$ of a manifold $M$ is called a codimension $q$ \textbf{leaf-like} submanifold, if around every point in $S$ there exists an open set $U$, a submersion $h:U\to \R^q$ and an at most countable set $\Lambda\subset \R^q$ such that $U\cap S=h^{-1}(\Lambda)$.
\end{definition}

By standard arguments (see e.g.\ \cite[Theorem 2.7]{Sharpe}), it can be shown that a codimension $q$ leaf-like submanifold is indeed an immersed submanifold of codimension $q$, and moreover that it is an initial submanifold, which implies that its smooth structure is uniquely determined. Note that, if the sets $\Lambda$ are assumed to be finite, then one obtains the usual notion of embedded submanifolds. By first applying the submersion theorem, the argument in the proof of Proposition \ref{prop : closed embedded} yield:
\begin{proposition}
A closed, connected, leaf-like submanifold is embedded.
\end{proposition}

The class of leaf-like submanifolds includes the leaves of singular foliations in the sense of Stefan \cite{Stefan} (see also \cite{Jan}), because these, by definition, admit local splittings. In particular, by the splitting theorem for Lie algebroids \cite[Theorem 1.1]{Rui}, the orbits of a Lie algebroid are also leaf-like submanifolds.

\end{document}